\documentclass[12pt,reqno]{amsproc}
		\usepackage[hyphens]{url} 
		\numberwithin{equation}{section}
		
		\makeatletter
\@namedef{subjclassname@2020}{%
			\textup{2020} Mathematics Subject Classification}
\makeatother
		
\usepackage{amsmath,amssymb,color}
\usepackage{amsthm} 
\usepackage{multicol}
\usepackage{pgfplots}
\usepackage{mathtools}
\usepackage{graphicx}
\usepackage{float}
\usepackage{caption}
\usepackage{subcaption}
\usepackage[bookmarksnumbered,colorlinks]{hyperref}
\usepackage{enumitem}
\mathtoolsset{showonlyrefs}
\usepackage[pagewise,mathlines]{lineno}
\usepackage{enumitem}
\usepackage{soul}
		
\theoremstyle{plain} 
\newtheorem{theorem}{\indent\sc Theorem}[section]
\newtheorem{lemma}[theorem]{\indent\sc Lemma}
\newtheorem{corollary}[theorem]{\indent\sc Corollary}

\theoremstyle{definition} 

\newtheorem{example}[theorem]{\indent\sc Example}

\title[A Liouville-type theorem for the $p$-Laplacian]{\sc A Liouville-type theorem for the $p$-Laplacian on complete non-compact Riemannian manifolds}
		
\bigskip
		
\author[F.R. dos Santos and M.N. Soares]{F\'abio R. dos Santos$^{\ast}$ and Matheus N. Soares}
		
\address{
Departamento de Matem\'atica \\
Universidade Federal de Pernambuco \\
50.740-540 Recife, Pernambuco \\
Brazil}
\email{fabio.reis@ufpe.br}
\email{matheus.nsoares@ufpe.br}
		
		
\thanks{$^{\ast}$Corresponding author}
		
\begin{document}
			
\begin{abstract}
A Liouville-type result for the $p$-Laplacian on complete Riemannian manifolds is proved. As an application are present some results concerning complete non-compact hypersurfaces immersed in a suitable warped product manifold.
\end{abstract}
			
\maketitle
		
\section{Introduction}

Let $(\Sigma^{n},g)$ be a Riemannian manifold. For any $1<p<\infty$ and any function $u\in W^{1,p}_{loc}(\Sigma)$, the $p$-Laplacian is the differential operator defined by,
\begin{equation}\label{eq:1.1}
\Delta_{p}u={\rm div}(\mid\nabla u\mid^{p-2}\nabla u).
\end{equation}
The $p$-Laplacian appears naturally on the variational problems associated to the energy functional $E_{p}:W^{1,p}_{0}(\Sigma)\to\mathbb{R}$ given by,
\begin{equation}\label{eq:1.2}
E_{p}(u)=\int_{M}\mid\nabla u\mid^{p}d\Sigma,
\end{equation}
where $d\Sigma$ denotes the element volume of $\Sigma^{n}$. In particular, if $p=2$, the $p$-Laplacian $\Delta_{p}$ reduces to the usual Laplace operator $\Delta_{2}=\Delta$. We denote by $\lambda_{1,p}(\Sigma)$ the first eigenvalue of $\Sigma^{n}$ which is defined by,
\begin{equation}\label{eq:1.3}
\lambda_{1,p}(\Sigma)=\inf\left\{\dfrac{\int_{\Sigma}\mid\nabla u\mid^{p}d\Sigma}{\int_{\Sigma}\mid u\mid^{p}d\Sigma};u\in W^{1,p}_{0}(\Sigma)\backslash\{0\}\right\}.
\end{equation}

Now, let us take $\Sigma^{n}$ an $n$-dimensional complete non-compact manifold and, let $\{\Omega_{k}\}_{k\in\mathbb{N}}$ be an exhaustion of $\Sigma^{n}$ by compact domains, that is, $\{\Omega_{k}\}_{k\in\mathbb{N}}$ are compact domains such that $\Sigma^{n}=\cup_{k=1}^{\infty}\Omega_{k}$ and $\Omega_{k}\subset\Omega_{k+1}$, for all $k\in\mathbb{N}$. We will consider the first eigenvalue $\lambda_{1,p}(\Omega_{k})$ of the following Dirichlet boundary value problem:
\begin{equation}\label{eq:1.4}
\left\{
\begin{array}{ccccc}
\Delta_{p}u &=& -\lambda\mid u\mid^{p-2}u, &\quad\mbox{in}\quad \Omega_{k}   \\
u &=& 0, &\quad\mbox{on}\quad \partial\Omega_{k}.
\end{array}
\right.
\end{equation}
The existence of the eigenvalue problem~\eqref{eq:1.4} and the variational characterization as in~\eqref{eq:1.3} were proved by Veron~\cite{Veron:91}. On the other hand, Lindqvist~\cite{Lindqvist:12} proved that $\lambda_{1,p}(\Omega_{k})$ is simple for each compact domain $\Omega_{k}$, $k\in\mathbb{N}$. By definition, we see that, $\lambda_{1,p}(\Sigma)=\lambda_{1,p}(\Omega_{k})$ for each compact domain $\Omega_{k}$, $k\in\mathbb{N}$. Using the domain monotonicity of $\lambda_{1,p}(\Omega_{k})$, we deduce that $\lambda_{1,p}(\Omega_{k})$ is non-increasing in $k\in\mathbb{N}$ and has a limit which is independent of the choice of the exhaustion of $\Sigma^{n}$. Therefore,
\begin{equation}\label{eq:1.5}
\lambda_{1,p}(\Sigma)=\lim_{k\to\infty}\lambda_{1,p}(\Omega_{k}).
\end{equation}

On the other hand, strongly $p$-subharmonic functions play an important role in the study of Riemannian manifolds. Let us recall that, a smooth function $u:\Sigma^{n}\to\mathbb{R}$ is said to be {\em strongly $p$-subharmonic} if $u$ satisfies the following differential inequality:
\begin{equation}\label{eq:1.6}
\Delta_{p}u\geq k>0.
\end{equation}
Concerning strongly $p$-subharmonic functions, we quote the results due Takegoshi~\cite{Takegoshi:01} for relations lying between the existence of a certain strongly $p$-subharmonic function and the volume growth property of the case $p\geq2$. Also, for $p=2$, Coghlan, Itokawa, and Kosecki~\cite{Coghlan:92} proved a Liouville-type result which said that every $2$-subharmonic function on a complete non-compact Riemannian manifold must be unbounded provided that its sectional curvature is bounded. A few years later, Leung~\cite{Leung:97} proved that the same result is valid by replacing the bounded sections curvature for the vanish first eigenvalue of $2$-Laplacian. As an application, Leung obtained an estimate for the size of the image set of some types of maps between Riemannian manifolds.

Recently, the authors~\cite{dos Santos:23} (see also~\cite{dos Santos:24,Santos:25.1,Santos:25.2}), introduced a divergence type operator, which extends the $p$-Laplacian, and developed Bochner and Reilly formulas for it. As an application, they obtained several lower estimates for the first eigenvalue of the $p$-Laplacian on compact Riemannian manifolds with or without boundary. Here, we are interest is to study the first eigenvalue of the $p$-Laplacian of complete non-compact Riemannian manifolds. More precisely, we obtain a characterization of $p$-subharmonic functions on a complete non-compact Riemannian manifold through the vanish first eigenvalue of the $p$-Laplacian. Proceeding with this picture, we obtain the following extension of the main result of~\cite[Theorem 1]{Leung:97} for the context of the $p$-Laplacian, for all $p\geq2$.
\begin{theorem}\label{teo:1.1}
If $\Sigma^{n}$ is a complete non-compact Riemannian manifold with $\lambda_{1,p}(\Sigma)=0$, then every strongly $p$-subharmonic function on $\Sigma^{n}$ is unbounded.
\end{theorem}

This manuscript is organized as follows: in Section~\ref{sec:setup} we present the basic background concerning the $p$-Laplacian as well as an important result which used to prove our main result (cf. Lemma~\ref{lem:2.1}). Next, in Section~\ref{sec:proof}, we start giving some examples and prove Theorem~\ref{teo:1.1} as well as a direct consequence (cf. Corollary~\ref{cor:3.4}). The last section is devoted to applications of the Theorem~\ref{teo:1.1} in the theory of isometric immersion. In this setting, we obtain results of nonexistence (cf. Theorems~\ref{teo:4.2}, \ref{teo:4.4} and Corollaries~\ref{cor:4.3}, \ref{cor:4.5}, \ref{cor:4.6}) and uniqueness (cf. Theorems~\ref{teo:4.6} and Corollaries~\ref{cor:4.7}, \ref{cor:4.8}) of hypersurfaces contained into a slab of certain warped product manifolds.

\section{A linear operator and key Lemma}\label{sec:setup}

By investigating lower estimates for the first eigenvalue of the $p$-Laplacian, Kawai and Nakauchi~\cite{Kawai:03} introduced an important linearization of the $p$-Laplacian. In fact, for every function $u\in\mathcal{C}^{1}(\Sigma)$, they defined the following functional:
\begin{equation}\label{eq:2.1}
\mathrm{P_{u}}(f)=\mid\nabla u\mid^{p-2}\Delta f+(p-2)\mid\nabla u\mid^{p-4}\langle{\rm Hess}f(\nabla u),\nabla u\rangle,\quad f\in\mathcal{C}^{2}(\Sigma),
\end{equation}
where ${\rm Hess}f$ denotes the Hessian of $f:\Sigma^{n}\to\mathbb{R}$.

As observed in~\cite{dos Santos:23}, if $u\in\mathcal{C}^{2}(\Sigma)$ with $\nabla u\neq0$, the functional $\mathrm{P_{u}}$ can be split as follows:
\begin{equation}\label{eq:2.2}
\mathrm{P_{u}}(f)=\mathrm{L_{u}}(f)+(p-2)\mid\nabla u\mid^{p-2}\mathrm{R_{u}}(f),
\end{equation}
with,
\begin{equation}\label{eq:2.3}
\mathrm{L_{u}}(f)={\rm div}(\mid\nabla u\mid^{p-2}\nabla f)\quad\mbox{and}\quad \mathrm{R_{u}}(f)=\langle\nabla f,\nabla\ln\mid\nabla u\mid\rangle-\mathrm{A_{u}}(f),
\end{equation}
where the functional $\mathrm{A_{u}}$ is defined by (see~\cite{dos Santos:23,Valtorta:14}),
\begin{equation}\label{eq:2.4}
\mathrm{A_{u}}(f)=\dfrac{\langle{\rm Hess}f(\nabla u),\nabla u\rangle}{\mid\nabla u\mid^{2}},\quad f\in\mathcal{C}^{2}(\Sigma).
\end{equation}
Moreover, by using~\eqref{eq:2.3},
\begin{equation}\label{eq:2.5}
\mathrm{L_{u}}(f)=\mid\nabla u\mid^{p-2}\Delta f+(p-2)\mid\nabla u\mid^{p-4}\langle{\rm Hess}u(\nabla u),\nabla f\rangle.
\end{equation}
In particular, for every $u\in\mathcal{C}^{2}(\Sigma)$, it follows from~\eqref{eq:1.1} and~\eqref{eq:2.4},
\begin{equation}\label{eq:2.5.1}
\mathrm{L_{u}}(u)=\Delta_{p}u=\mid\nabla u\mid^{p-2}\left(\Delta u+(p-2)\mathrm{A_{u}}(u)\right). 
\end{equation}
A standard computation allows to check that $\mathrm{L_{u}}$ satisfy Leibniz's rule and is elliptic for $p\geq2$.

By using this previous machinery, the next result is an extension of~\cite[Lemma 1]{Palmer:90} due Palmer to the $p$-Laplacian for $p\geq2$ (see also~\cite{Leung:92}).
\begin{lemma}\label{lem:2.1}
Let $D$ be a relatively compact smoothly bounded domain on a Riemannian manifold $\Sigma^{n}$. Let $\lambda_{1,p}(D)$ denote the first eigenvalue of the problem,
\begin{equation}
\begin{cases}
\Delta_p u+\lambda\mid u\mid^{p-2}u=0 & \text { in } D, \\ u=0 & \text { on } \partial D.
\end{cases}
\end{equation}
Suppose $f$ is a smooth function on $\overline{D}$ that satisfies $\Delta_p f \geq 1$ in $D$. Then,
\begin{equation}
\lambda_{1,p}(D)\geq\frac{1}{(\beta-\alpha)^{p-1}},
\end{equation}
where $\alpha, \beta$ are any lower and upper bounds respectively of $f$ on $\overline{D}$.
\end{lemma}

\begin{proof}
First of all, let us consider $u$ the first eigenfunction associated with $\lambda_{1,p}(D)$, and assume, without loss of generality, that $u>0$. We consider $a$ any constant such as $\beta<a$. Let us define then the function $v=\frac{u}{a-f}$. We observe that $v$ attains the maximum at some point $x_0\in\text{int}(D)$ since $v\equiv 0$ in $\partial D$ and $a-f>0$. Hence in $x_{0}$,
\begin{equation}\label{eq:2.6}
\nabla v (x_0)=0\quad\mbox{and}\quad \mathrm{L_u} (v)(x_0)\leq 0.
\end{equation}
	
A direct computation shows that the gradient of $v$ is given by,
\begin{equation}\label{eq:2.7}
\nabla v=\dfrac{1}{a-f}\nabla u+\dfrac{u}{(a-f)^{2}}\nabla f
\end{equation}
and that, for all $X\in\mathfrak{X}(\Sigma)$, the hessian of $v$ is
\begin{equation}\label{eq:2.8}
\begin{split}
\nabla_{X}\nabla v&=\dfrac{1}{(a-f)^{2}}\langle\nabla f,X\rangle\nabla u+\dfrac{1}{a-f}\nabla_{X}\nabla u+\dfrac{u}{(a-f)^{2}}\nabla_{X}\nabla f\\
&\quad+\dfrac{1}{(a-f)^{3}}\left((a-f)\langle\nabla u,X\rangle+2u\langle X,\nabla f\rangle\right)\nabla f.
\end{split}
\end{equation}
By tracing~\eqref{eq:2.8}, we have
\begin{equation}\label{eq:2.9}
\begin{split}
\Delta v&=\dfrac{1}{a-f}\Delta u+\dfrac{u}{(a-f)^{2}}\Delta f+\dfrac{2u}{(a-f)^{3}}\mid\nabla f\mid^{2}+\dfrac{2}{(a-f)^{2}}\langle\nabla u,\nabla f\rangle.
\end{split}
\end{equation}
Moreover, from~\eqref{eq:2.4} and~\eqref{eq:2.7},
\begin{equation}\label{eq:2.10}
\langle{\rm Hess}u(\nabla u),\nabla v\rangle=\dfrac{1}{a-f}\langle{\rm Hess}u(\nabla u),\nabla u\rangle+\dfrac{u}{(a-f)^{2}}\langle{\rm Hess}u(\nabla u),\nabla f\rangle.
\end{equation}
Hence, by inserting~\eqref{eq:2.7} and~\eqref{eq:2.10} in~\eqref{eq:2.5}, we get
\begin{equation}\label{eq:2.11}
\begin{split}
\mathrm{L_u}(v)&=\dfrac{1}{a-f}\Delta_{p}u+\dfrac{u}{(a-f)^{2}}\mathrm{L_{u}}(f)+\dfrac{2u}{(a-f)^{3}}\mid\nabla f\mid^{2}\mid\nabla u\mid^{p-2}\\
&\quad+\dfrac{2}{(a-f)^{2}}\langle\nabla u,\nabla f\rangle\mid\nabla u\mid^{p-2}.
\end{split}
\end{equation}
	
On the other hand, once $\nabla v(x_0)=0$, \eqref{eq:2.7} gives, $\nabla u=-\dfrac{u}{a-f}\nabla f$ in $x_0$. Hence, in $x_{0}$,
\begin{equation}\label{eq:2.12}
{\rm Hess}u(\nabla u)=\dfrac{u^2}{(a-f)^{2}}{\rm Hess}f(\nabla f),
\end{equation}
and consequently
\begin{equation}\label{eq:2.13}
\begin{split}
\mathrm{L_{u}}(f)&=\mid\nabla u\mid^{p-2}\Delta f+(p-2)\mid\nabla u\mid^{p-4}\langle{\rm Hess}u(\nabla u),\nabla f\rangle\\
&=\left(\dfrac{u}{a-f}\right)^{p-2}\mid\nabla f\mid^{p-2}\Delta f+(p-2)\left(\dfrac{u}{a-f}\right)^{p-2}\mid\nabla f\mid^{p-2}A_{f}(f)\\
&=\left(\dfrac{u}{a-f}\right)^{p-2}\Delta_{p}f.
\end{split}
\end{equation}
Besides this,
\begin{equation}\label{eq:2.14}
\begin{split}
\dfrac{2u}{(a-f)^{3}}\mid\nabla f\mid^{2}+\dfrac{2}{(a-f)^{2}}\langle\nabla u,\nabla f\rangle=0.
\end{split}
\end{equation}
Therefore, from~\eqref{eq:2.11}, ~\eqref{eq:2.13} and~\eqref{eq:2.14} we obtain in $x_{0}$,
\begin{equation}\label{eq:2.15}
0\geq \mathrm{L_{u}}(v)=\dfrac{1}{a-f}\Delta_{p}u+\dfrac{u^{p-1}}{(a-f)^{p}}\Delta_{p}f.
\end{equation}
Since $\Delta_{p}u=-\lambda_{1,p}(D)\mid u\mid^{p-2}u$, $u(x_{0})>0$, $\alpha<f(x_{0})<a$, and $\Delta_{p}f\geq1$, \eqref{eq:2.15} implies
\begin{equation}\label{eq:2.16}
\lambda_{1,p}(D)\geq\dfrac{1}{(a-\alpha)^{p-1}}.
\end{equation}
By taking $a\to\beta$ we obtain the desired.
\end{proof}

\section{Proof of Theorem~\ref{teo:1.1}}\label{sec:proof}

Before presenting the proof of our main result, we will give some examples and classical results concerning complete Riemannian manifolds having vanished the first $p$-Laplacian eigenvalue. The first one is (cf. \cite{do Carmo:99})
\begin{example}\label{ex:3.1}
Let $\Sigma^2=(\mathbb{R}^2, ds^2),$ with $ds^2=dr^2+g^2(r)d\theta^2$, where $g(r)$ is a nonnegative $\mathcal{C}^2([0,+\infty))$ function with $g(0)=0$, $g(r)>0$ for $0<r\leq 1$ and $g(r)=e^{-r}$ for $r>1$. Note that ${\rm vol}(\Sigma)<+\infty$ and hence $\lambda_{1,p}(\Sigma)=0$.
\end{example}

Let $B_{r}(q)\subset\Sigma^{n}$ be the geodesic ball centered at $q$ and with radius $r$. We say that the volume of $\Sigma^{n}$ has polynomial growth if there exist positive numbers $a$ and $c$ such that ${\rm vol}(B_{r}(q))\leq cr^{a}$. In this setting,  the next two results are due Batista et. al~\cite{Batista:14}, extends the classical result due do Carmo~\cite{do Carmo:99} for all $p>2$. The first one assume that the Riemannian manifold has infinite volume:
\begin{example}\label{ex:3.2} Let $\Sigma^n$ be an open manifold with infinite volume and $\Omega$ an arbitrary compact set of $\Sigma^{n}$. If $\Sigma^{n}$ has polynomial volume growth, then $\lambda_{1,p}(M\backslash\Omega)=0$.
\end{example}

The second assume that the Riemannian manifold is complete and non-compact:
\begin{example}\label{ex:3.3} 
Let $\Sigma^{n}$ be a complete non-compact Riemannian manifold with polynomial volume growth, then $\lambda_{1,p}(\Sigma)=0$.
\end{example}

Now, we are able to proof our main result.
\begin{proof}[\underline{Proof of Theorem~\ref{teo:1.1}}]
	
Let us suppose that $u\in\mathcal{C}^{2}(\Sigma)$ is bounded. Since $u$ is a strongly $p$-subharmonic function, from~\eqref{eq:1.6}
\begin{equation}\label{eq:3.1}
\Delta_{p}u\geq k>0,
\end{equation}
for some positive constant $k$. Now, we consider the function $f:\Sigma^{n}\to\mathbb{R}$ given by
\begin{equation}\label{eq:3.2}
f=\dfrac{u}{k^{\frac{1}{p-1}}}.
\end{equation}
From this, direct computation gives
\begin{equation}\label{eq:3.3}
\nabla f=\dfrac{1}{k^{\frac{1}{p-1}}}\nabla u\quad\mbox{and}\quad A_{f}(f)=\dfrac{1}{k^{\frac{1}{p-1}}}\mathrm{A_{u}}(u).
\end{equation}
Hence, by inserting~\eqref{eq:3.3} in~\eqref{eq:2.5.1}, it follows from~\eqref{eq:3.1}
\begin{equation}\label{eq:3.4}
\begin{split}
\Delta_{p}f&=\mid\nabla f\mid^{p-2}\left(\Delta f+(p-2)A_{f}(f)\right)\\
&=\left(\dfrac{1}{k^{\frac{1}{p-1}}}\right)^{p-2}\dfrac{1}{k^{\frac{1}{p-1}}}\left(\Delta u+(p-2)\mathrm{A_{u}}(u)\right)\\
&=\dfrac{1}{k}\Delta_{p}u\geq1.
\end{split}
\end{equation}
On the other hand, from~\eqref{eq:3.2} we see that $f$ is bounded. Thus, there exists positive constant $\alpha$ and $\beta$ such that $\alpha<f<\beta$. Therefore, by applying Lemma~\ref{lem:2.1},
\begin{equation}\label{eq:3.5}
\lambda_{1,p}(\Sigma)\geq\dfrac{1}{(\beta-\alpha)^{p-1}}>0,
\end{equation}
which contradicts the fact that the first eigenvalue is zero.
\end{proof}

As an immediate application, it follows from Example~\ref{eq:3.3},
\begin{corollary}\label{cor:3.4}
If $\Sigma^{n}$ is a complete non-compact Riemannian manifold with polynomial volume growth, then every strongly $p$-subharmonic function on $\Sigma^{n}$ is unbounded.
\end{corollary}

\section{Hypersurfaces in warped product manifolds}\label{sec:applications}

Let $(M^{n},\langle,\rangle_{M})$ be a connected $n$-dimensional $(n\geq2)$ oriented Riemannian manifold and $f:I\subset\mathbb{R}\to\mathbb{R}_{+}$ a positive smooth function. In the product differentiable manifold $I\times M^{n}$. A particular class of Riemannian manifolds is the one obtained by endowing $I\times M^{n}$ with the metric
\begin{equation}\label{wp}
\langle v,w\rangle_{(t,q)}=\langle(\pi_I)_*v,(\pi_I)_*w\rangle_{I}+(f\circ\pi_{I})^{2}\langle(\pi_M)_*v,(\pi_M)_*w\rangle_{M},
\end{equation}
with $(t,q)\in I\times M^{n}$ and $v,w\in T_{(t,q)}(I\times M)$, where $\pi_{I}$ and $\pi_{M}$ denote the projections onto the corresponding factor. Such a space is called a {\em warped product} and $f$ the {\em warped function}, and in what follows, we shall write $I\times_{f}M^{n}$ to denote it. In particular the family of hypersurfaces $M^{n}_{t}=\{t\}\times M^{n}$ (called slices) form a foliation $t\in I\to M^{n}_{t}$ of $I\times_{f}M^{n}$ by totally umbilical leaves of constant mean curvature
\begin{equation}\label{eq:4.8}
\mathcal{H}(t)=\dfrac{f'(t)}{f(t)}=(\ln f)'(t),
\end{equation}
with respect $-\partial_{t}$, where
\begin{equation}\label{eq:4.9}
\partial_{t}:=(\partial/\partial_{t})\mid_{(p,t)},\quad(t,q)\in I\times M^{n}
\end{equation}
is a conformal and unitary vector field, that is,
\begin{equation}\label{eq:4.8}
\overline{\nabla}_{X}\partial_{t}=\mathcal{H}(t)(X-\langle X,\partial_{t}\rangle\partial_{t})\quad\mbox{and}\quad\langle\partial_{t},\partial_{t}\rangle=1,
\end{equation}
for all $X\in\mathfrak{X}(\Sigma)$.

Let $x:\Sigma^{n}\to I\times_{f}M^{n}$ is an isometrically immersed connected hypersurface in the warped product manifold $I\times_{f}M^{n}$ with second fundamental $A$ with respect to the normal direction $N$. We define the {\em height function} $h:\Sigma^{n}\to\mathbb{R}$ by setting $h(q)=\langle x(q),\partial_{t}\rangle$. A direct computation shows that the gradient of $h$ on $\Sigma^{n}$ is
\begin{equation}\label{eq:4.10}
\nabla h=\partial_t^{\top}=\partial_t-\langle N,\partial_{t}\rangle N,
\end{equation}
where $(\cdot)^{\top}$ denotes the tangential component of a vector field on $I\times_{f}M^{n}$ along $\Sigma^{n}$ and $\langle N,\partial_t\rangle$ is the {\em angle function}. Thus, we get
\begin{equation}\label{eq:4.11}
\mid\nabla h\mid^2=\mid\partial_{t}^{\top}\mid^2=1-\langle N,\partial_t\rangle^2,
\end{equation}
where $\mid\cdot\mid$ denotes the norm of a vector field on $\Sigma^{n}$.

As a sub-product of the digression above, we have the $p$-Laplacian version of~\cite[Proposition 6]{Alias:13}.
\begin{lemma}\label{lem:4.1}
Let $x:\Sigma^{n}\to I\times_{f}M^{n}$ be an isometric immersion into a warped product manifold. Define:
\begin{equation}\label{eq:4.16.0}
\sigma(h)=\int_{t_{0}}^{h(\cdot)}f(u)du,
\end{equation}
where $h$ is the height function. If $q\in\Sigma^{n}$ is a point contained in a domain $U$ and $\nabla h\neq0$ on $U$, then
\begin{equation}\label{eq:4.16}
\begin{split}
\Delta_{p}h&=\langle N,\partial_{t}\rangle\left(nH+(p-2)\mid\nabla h\mid^{-2}\langle A(\nabla h),\nabla h\rangle\right)\mid\nabla h\mid^{p-2}\\
&\quad+\mathcal{H}(h)\left(n+p-2-(p-1)\mid\nabla h\mid^{2}\right)\mid\nabla h\mid^{p-2},
\end{split}
\end{equation}
and
\begin{equation}\label{eq:4.16.5}
\begin{split}
\Delta_{p}\sigma(h)&=f(h)^{p-1}\langle N,\partial_{t}\rangle\left(nH+(p-2)\mid\nabla h\mid^{-2}\langle A(\nabla h),\nabla h\rangle\right)\mid\nabla h\mid^{p-2}\\
&\quad+(n+p-2)\mathcal{H}(h)f(h)^{p-1}\mid\nabla h\mid^{p-2},
\end{split}
\end{equation}
for all $p\in(1,+\infty)$.
\end{lemma}

\begin{proof}
From Gauss and Weingarten formulas, the Hessian of the height function and the gradient of the angle function satisfy:
\begin{equation}\label{eq:4.12}
\nabla_{X}\nabla h=\langle N,\partial_{t}\rangle A(X)+\mathcal{H}(h)\left(X-\langle X,\nabla h\rangle\nabla h\right),
\end{equation}
and
\begin{equation}\label{eq:4.13}
X\langle N,\partial_{t}\rangle=-\langle A(\nabla h),X\rangle-\mathcal{H}(h)\langle N,\partial_{t}\rangle\langle\nabla h,X\rangle,
\end{equation}
for every $X\in\mathfrak{X}(\Sigma)$. Consequently,
\begin{equation}\label{eq:4.14}
\Delta_{2}h=n\langle N,\partial_{t}\rangle H+\mathcal{H}(h)(n-\mid\nabla h\mid^{2}),
\end{equation}
and
\begin{equation}\label{eq:4.15}
A_{h}(h)=\mid\nabla h\mid^{-2}\langle N,\partial_{t}\rangle\langle A(\nabla h),\nabla h\rangle+\mathcal{H}(h)\left(1-\mid\nabla h\mid^{2}\right).
\end{equation}
Hence, by inserting~\eqref{eq:4.14} and~\eqref{eq:4.15} in~\eqref{eq:2.5.1}, we get
\begin{equation}\label{eq:4.16.1}
\begin{split}
\Delta_{p}h&=\langle N,\partial_{t}\rangle\left(nH+(p-2)\mid\nabla h\mid^{-2}\langle A(\nabla h),\nabla h\rangle\right)\mid\nabla h\mid^{p-2}\\
&\quad+\mathcal{H}(h)\left(n+p-2-(p-1)\mid\nabla h\mid^{2}\right)\mid\nabla h\mid^{p-2}.
\end{split}
\end{equation}
On the other hand, since $\nabla\sigma(h)=f(h)\nabla h$, we have
\begin{equation}\label{eq:4.16.2}
\nabla_{X}\nabla\sigma(h)(X)=f(h)\nabla_{X}\nabla h+f'(h)\langle X,\nabla h\rangle\nabla h,
\end{equation}
and
\begin{equation}\label{eq:4.16.3}
A_{\sigma(h)}(\sigma(h))=f(h)A_{h}(h)+f'(h)\mid\nabla h\mid^{2}.
\end{equation}
Hence
\begin{equation}\label{eq:4.16.4}
\Delta_{p}\sigma(h)=f(h)^{p-1}\left(\Delta_{p}h+(p-1)\mathcal{H}(h)\mid\nabla h\mid^{p}\right).
\end{equation}
So, by inserting~\eqref{eq:4.16.1} in~\eqref{eq:4.16.4}, we conclude that
\begin{equation}
\begin{split}
\Delta_{p}\sigma(h)&=f(h)^{p-1}\langle N,\partial_{t}\rangle\left(nH+(p-2)\mid\nabla h\mid^{-2}\langle A(\nabla h),\nabla h\rangle\right)\mid\nabla h\mid^{p-2}\\
&\quad+(n+p-2)\mathcal{H}(h)f(h)^{p-1}\mid\nabla h\mid^{p-2},
\end{split}
\end{equation}
for all $p\in(1,+\infty)$. 
\end{proof}

From now on, we will deal with hypersurface contained in a {\em slab} of $I\times_{f}M^{n}$ which means that $\Sigma^{n}$ lies between two leaves $M^{n}_{t_1}$, $M^{n}_{t_2}$ with $t_1<t_2$ of the foliation $M^{n}_{t}$; in other words, $\Sigma^{n}$ is contained in a bounded region of the type
\begin{equation}\label{slab}
[t_{1},t_{2}]\times M^{n}=\{(t,q)\in I\times_{f}M^{n};t_{1}\leq t\leq t_{2}\mbox{and} q\in M^{n}\}.
\end{equation}
It is clear that if the $\Sigma^{n}$ is contained in the region~\eqref{slab}, the height function satisfies: $t_{1}\leq h(q)\leq t_{2}$ for all $q\in\Sigma^{n}$. Using this notation, Alías and Dajczer~\cite{Alias:06} investigated complete surfaces properly immersed in a slab of $I\times_{f}M^{n}$ under suitable geometric assumptions on the Riemannian fiber $M^{n}$. Our first result generalizes this one without any assumption on $M^{n}$.
\begin{theorem}\label{teo:4.2}
There exists no complete non-compact immersed hypersurface $\Sigma^{n}$ contained in a slab of $I\times_{f}M^{n}$ having $\lambda_{1,2}(\Sigma)=0$ and mean curvature satisfying 
\begin{equation}\label{hip1}
\sup_{\Sigma}\mid H\mid<\min_{[t_{1},t_{2}]}\mathcal{H}(t).
\end{equation}
\end{theorem}

\begin{proof}
Suppose there exists a complete non-compact immersed hypersurface $\Sigma^{n}$ into the slab of $I\times_{f}M^{n}$, such that $\lambda_{1,2}(\Sigma)=0$ and that the mean curvature satisfy~\eqref{hip1}. By taking $p=2$ in~\eqref{eq:4.16.5},
\begin{equation}\label{eq:4.16.6}
\Delta_{2}\sigma(h)=nf(h)\left(\langle N,\partial_{t}\rangle H+\mathcal{H}(h)\right).
\end{equation}
From Cauchy-Schwarz' inequality and~\eqref{hip1},
\begin{equation}\label{eq:4.16.6.1}
\langle N,\partial_{t}\rangle H+\mathcal{H}(h)\geq-\sup_{\Sigma}\mid H\mid+\inf_{\Sigma}\mathcal{H}(h)>0.
\end{equation}
Hence
\begin{equation}\label{eq:4.16.7}
\Delta_{2}\sigma(h)\geq n\inf_{\Sigma}f(h)\left(-\sup_{\Sigma}\mid H\mid+\inf_{\Sigma}\mathcal{H}(h)\right)>0.
\end{equation}
Thus, $\sigma(h)$ is a strongly $2$-subharmonic function on $\Sigma^{n}$. Therefore, from Theorem~\ref{teo:1.1}, $\sigma(h)$ must be unbounded, which contradicts the fact the $\Sigma^{n}$ is contained into a slab.
\end{proof}

As application of Corollary~\ref{cor:3.4}, we get the following consequence of Theorem~\ref{teo:4.2}:
\begin{corollary}\label{cor:4.3}
There exists no complete non-compact immersed hypersurface $\Sigma^{n}$ contained in a slab of $I\times_{f}M^{n}$ with polynomial volume growth and mean curvature satisfying 
\begin{equation}
\sup_{\Sigma}\mid H\mid<\min_{[t_{1},t_{2}]}\mathcal{H}(t).
\end{equation}
\end{corollary}

Now, by assuming that the warped product manifold $I\times_{f}M^{n}$ satisfies $\mathcal{H}(t)\geq1$ for all $t\in I$, we obtain
\begin{theorem}\label{teo:4.4}
There exists no complete non-compact immersed hypersurface $\Sigma^{n}$ contained in a slab of $I\times_{f}M^{n}$ having $\lambda_{1,2}(\Sigma)=0$ and mean curvature satisfying $\sup_{\Sigma}H<1$.
\end{theorem}

\begin{proof}
Let us suppose, by contradiction, the existence of a such hypersurface. From~\eqref{eq:4.11} we can easily see that $n-\mid\nabla h\mid^{2}>0$. Then, \eqref{eq:4.14} can be write as follows,
\begin{equation}\label{eq:4.17}
\Delta_{2}h\geq n\langle N,\partial_{t}\rangle H+n-\mid\nabla h\mid^{2},
\end{equation}
provided that $\mathcal{H}\geq1$.

By using $\varepsilon$-Young's inequality $2ab\leq\varepsilon a^{2}+\varepsilon^{-1}b^{2}$, $\varepsilon>0$, for
\begin{equation*}
a=\sqrt{n}\mid\langle N,\partial_{t}\rangle\mid\quad\mbox{and}\quad b=\sqrt{n}\mid H\mid
\end{equation*}
we have
\begin{equation}\label{eq:4.18}
\begin{split}
\Delta_{2}h&\geq-n\mid\langle N,\partial_{t}\rangle\mid\mid H\mid+n-\mid\nabla h\mid^{2}\\
&\geq\dfrac{n}{2}\left(2-\varepsilon+\left(\varepsilon-\dfrac{2}{n}\right)\mid\nabla h\mid^{2}-\dfrac{1}{\varepsilon}H^{2}\right).
\end{split}
\end{equation}
By considering $\varepsilon=1$, we get
\begin{equation}\label{eq:4.19}
\Delta_{2}h\geq\dfrac{n}{2}\left(1+\dfrac{n-2}{n}\mid\nabla h\mid^{2}-H^{2}\right)\geq\dfrac{n}{2}\left(1-\sup_{\Sigma}H^{2}\right)>0.
\end{equation}
Which means that $h$ is strongly $2$-subharmonic. Therefore, we are in position to apply Theorem~\ref{teo:1.1} to infer that $h$ is unbounded, a contradiction.
\end{proof}

\begin{corollary}\label{cor:4.5}
There exists no complete non-compact immersed hypersurface $\Sigma^{n}$ contained in a slab of $I\times_{f}M^{n}$ having polynomial volume growth and mean curvature satisfying $\sup_{\Sigma}H<1$.
\end{corollary}

Following~\cite{Tashiro:65}, when the warped function $f$ is the exponential, the product warped manifold $\mathbb{R}\times_{e^{t}}M^{n}$ belongs to a class of manifolds known as {\em pseudo-hyperbolic space}. In fact, this terminology comes from the observation that the hyperbolic space, $\mathbb{H}^{n+1}$, can be described as a warped product, $\mathbb{R}\times_{e^{t}}\mathbb{R}^{n}$. In this structure, the slices represent the set of all horospheres that share the same fixed point on the asymptotic boundary, $\partial\mathbb{H}^{n+1}$. Together, these horospheres provide a complete foliation of $\mathbb{H}^{n+1}$ (see also~\cite{Alias:06,Montiel:99}). In this setting, we have the following extension of~\cite[Corollary 3]{Alias:06}.
\begin{corollary}\label{cor:4.6}
There exists no complete non-compact immersed hypersurface $\Sigma^{n}$ contained in a slab of pseudo-hyperbolic manifold $\mathbb{R}\times_{e^{t}}M^{n}$ having $\lambda_{1,2}(\Sigma)=0$ and mean curvature satisfying $\sup_{\Sigma}H<1$.
\end{corollary}

For the next result, we will deal with {\em helix-type} hypersurfaces or having {\em constant angle}. We say that a hypersurface $\Sigma^{n}$ of $I\times_{f}M^{n}$ is a helix-type hypersurface if the angle function $\langle N,\partial_{t}\rangle$ is constant on $\Sigma^{n}$. In this setting, by considering that the warped product manifold $I\times_{f}M^{n}$ satisfies $\mathcal{H}(t)\geq1$ for all $t\in I$, we obtain.
\begin{theorem}\label{teo:4.6}
Let $\Sigma^{n}$ be a complete non-compact helix-type hypersurface contained in a slab of $I\times_{f}M^{n}$ with $\lambda_{1,p}(\Sigma)=0$, $p>2$. If the mean curvature (not necessary constant) satisfies $H^{2}\leq1$, then $\Sigma^{n}$ is a slice.
\end{theorem}

\begin{proof}
We will assume for contradiction, that $\Sigma^{n}$ is not a slice of $I \times_{f}M^{n}$ Since $\Sigma^{n}$ is a helix-type hypersurface which is not a slice, we must have $\langle N,\partial_{t}\rangle\neq\pm1$, and hence $\mid\nabla h\mid=const.>0$. From~\eqref{eq:4.11} and~\eqref{eq:4.13}, we have
\begin{equation}\label{eq:4.20}
\mid\nabla h\mid^{-2}\langle N,\partial_{t}\rangle\langle A(\nabla h),\nabla h\rangle=-\mathcal{H}(h)\langle N,\partial_{t}\rangle^{2}=-\mathcal{H}(h)(1-\mid\nabla h\mid^{2}).
\end{equation}
By replacing this in~\eqref{eq:4.16},
\begin{equation}
\Delta_{p}h=\left(n\langle N,\partial_{t}\rangle H+\mathcal{H}(h)(n-\mid\nabla h\mid^{2})\right)\mid\nabla h\mid^{p-2},
\end{equation}
that is,
\begin{equation}\label{eq:4.21}
\Delta_{p}h=\mid\nabla h\mid^{p-2}\Delta_{2}h.
\end{equation}
Hence, from~\eqref{eq:4.19},
\begin{equation}
\Delta_{p}h\geq\left(\dfrac{n-2}{2}\right)\mid\nabla h\mid^{p}=const.>0.
\end{equation}
Therefore, $h$ is a strongly $p$-subharmonic function. Since $\lambda_{1,p}(\Sigma)=0$, it follows from Theorem~\ref{teo:1.1} that $h$ is unbounded, a contradiction.
\end{proof}

\begin{corollary}\label{cor:4.7}
Let $\Sigma^{n}$ be a complete non-compact helix-type hypersurface contained in a slab of $I\times_{f}M^{n}$ with polynomial volume growth. If the mean curvature (not necessary constant) satisfies $H^{2}\leq1$, then $\Sigma^{n}$ is a slice.
\end{corollary}

In the inspirit of Corollary~\ref{cor:4.5}, we close this manuscript with the following extension of~\cite[Theorem 4]{Alias:06} and~\cite[Theorem 3.3]{de Lima:14}.
\begin{corollary}\label{cor:4.8}
Let $\Sigma^{n}$ be a complete non-compact helix-type hypersurface contained in a slab of pseudo-hyperbolic manifold $\mathbb{R}\times_{e^{t}}M^{n}$ with $\lambda_{1,p}(\Sigma)=0$. If the mean curvature (not necessary constant) satisfies $H^{2}\leq1$, then $\Sigma^{n}$ is a slice.
\end{corollary}


\section*{Acknowledgements}

The first author is partially supported by CNPq, Brazil, grant 311124/2021-6 and Propesqi (UFPE). The second author is partially supported by CNPq, Brazil.


\end{document}